\documentclass[a4paper]{amsart}
\usepackage[utf8]{inputenc}
\usepackage{amsmath,amssymb,amsthm}
\usepackage{units}
\usepackage{bm, comment, esint}
\usepackage[T1]{fontenc}
\usepackage{srcltx,
}
\def\avint[#1]{\mathop{\,\rlap{-}\!\!\!\int_{#1}}\nolimits}
\def \w{{\bf w}}
\def \vv{{\bf v}}
\def \f {{\bf f}}
\def \S {{\bf S}}
\def \K {{\bf K}}
\def \A {{\bf A}}
\def \DD {{\bf D}}
\def \B {{\bf B}}
\def \C {{\bf C}}

\def \iO{\int_{\Omega}}
\def \R{\mathbb{R}}
\def \N{\mathbb{N}}
\def \O{\Omega}
\def \na{\nabla}
\def \q{{\bf q}}

\def\be{\begin{equation}}
\def\ba{\begin{array}}
\def\ea{\end{array}}
\def\ee{\end{equation}}

\def\displ{\displaystyle \vspace{6pt}}

\newcommand{\cD}{{\bf \mathcal{D} v}}
\def \Scv{\S(c, \cD)}
\def \Scvn{{\S(c^{n,m}, \cD^{n,m})}}
\def \D {\mathcal{D}}
\def \DD{{\bf D}}
\def \iR {\int_{B(x_0,R)}}
\def \iRR {\int_{A_R}}
\def \iRt {\int_{B(x_0,\frac{R}{2})}}
\def \imR {\fint_{B(x_0,R)}}
\def \imRt {\fint_{B(x_0,\frac{R}{2})}}
\def \pa{\partial}

\newcommand{\Div}{\operatorname{div}}
\newcommand{\Per}{\operatorname{per}}

\numberwithin{equation}{section}

\font\sc=cmcsc10
\newtheorem{defi}
{Definition}
\newtheorem{tho}
{Theorem}

\newtheorem{lem}
{Lemma}

\newcommand{\tha}{\theta_A}
\newcommand{\g}{\mathbf g}
\newcommand{\meanless}[2]{(#1)_{0,#2}}
\newcommand{\diver}{\operatorname{div}}

\begin{document}

\title[Classical solution to concentration dependent power-law fluid]{On the existence of classical solution to the steady flows of generalized Newtonian fluid with concentration dependent power-law index}

\author[A.~Abbatiello]{Anna Abbatiello}
\address{Dipartimento di Matematica e Fisica,  Universit\`{a} degli Studi della Campania  ``L. Vanvitelli", viale Lincoln, 81100 Caserta, Italy}
\email{anna.abbatiello@unicampania.it}

\author[ M.~Bul\'{\i}\v{c}ek]{Miroslav Bul\'{\i}\v{c}ek}
\address{Charles University, Faculty of Mathematics and Physics, Mathematical Institute, Sokolovsk\'{a} 83, 186~75, Prague, Czech Republic}
\email{mbul8060@karlin.mff.cuni.cz}

\author[P.~Kaplick\'{y}]{Petr Kaplick\'{y}}
\address{Charles University, Faculty of Mathematics and Physics, Mathematical Institute, Sokolovsk\'{a} 83, 186~75, Prague, Czech Republic}
\email{kaplicky@karlin.mff.cuni.cz}

\thanks{A.~Abbatiello is partially supported by National Group of Mathematical Physics (GNFM-INdAM)  via GNFM Progetto Giovani 2017. A.~Abbatiello is also grateful to Charles University for the hospitality during her visit when the work was performed. M.~Bul\'{\i}\v{c}ek's work was supported by the Czech Science Foundation (grant no. 16-03230S). M.~Bul\'{\i}\v{c}ek and P.~Kaplick\'{y} are members of Ne\v{c}as Center for Mathematical Modeling.}


\begin{abstract}
Steady flows of an incompressible homogeneous chemically reacting fluid are described by a coupled system, consisting of the generalized  Navier--Stokes equations and convection - diffusion equation with  diffusivity dependent on the concentration and the shear rate.  Cauchy stress behaves like power-law fluid with the exponent depending on the concentration. We prove the existence of a classical solution for the two dimensional periodic case whenever the power law exponent is above one and less than infinity.
\end{abstract}

\maketitle

%\section*{Notes}
%
%\par
%We may reprove the estimate of Bulicek Pustejovska for the approximated system in order to get \emph{good estimate in a bad space} in dimensions 2 and 3. Here I chose to proceed in typical 2D way following the new results of Anna.
%
%\par
%At several places the spaces of functions with zero trace appear. This should be changed to spaces of periodic functions with zero mean.
%
%\par
%For periodic boundary conditions the maximum principle is not applicable. We will need to use Moser iteration technique to get boundedness of $c^n$. Do we need estimates uniform in $A$?
%
%\par
%For periodic bc. the convective term disappears when testing \cite{syst-n2} with $c$ and we get that $c=0$ everywhere, i.e. $p(c)=p(0)$. To make the problem nontrivial we need to add some forcing to equation for $c$. What is its correct/physical form? I add just net applied force $g$.
%
%\par
%When deriving that $\nabla c^n\in L^{2+\epsilon}$ the estimate of the terms with $g$ is still missing.
%
%\par
%We need to assume that $p$ is at least Lipschitz, see \eqref{c22}.
\section{Introduction}\label{S1}
%\section{Definition of the problem}

The main goal of the paper is to prove the existence of a classical solution to a class of models describing the steady flow of an incompressible homogeneous chemically reacting fluid. More specifically, we are interested in regularity properties of  the velocity $\vv :\O \rightarrow \R^d$, the pressure  $\pi : \O \rightarrow \R$
and the concentration distribution $c: \O \rightarrow \R_+$ that solve the following system of partial differential equations
\be\label{problem}
\begin{split}
\Div (\vv \otimes \vv)- \Div \S(c,\cD) &= -\na \pi + \f,\\
\Div \vv &=0,\\
\Div  {c \vv} - \Div \q_c(c, \na c, \cD)&= -\Div \g
\end{split}
\ee
in a domain $\O \subset \R^{d}$. Here $\f :\O \rightarrow \R^d $ represents a given density of the volume forces, $\cD :\O \rightarrow \R^{d\times d}$ denotes the symmetric part of the velocity gradient $\na \vv$ and $\S(c,\cD):\O \rightarrow \R^{d\times d}$ is the constitutively determined part of the Cauchy stress tensor, $\q_c(c, \na c, \cD):\O \rightarrow \R^d$ is the concentration flux and $\g :\O \rightarrow \R^d$ represents a source term for the chemical concentration. The first equation in \eqref{problem} is the balance of linear momentum, the second one is the incompressibility constraint and the last equation (convection--diffusion--reaction) describes the conservation of the chemical concentration. Although, we will be interested only in two dimensional results, i.e., for $d=2$, we keep the notation for general dimension in this introductory part in order to comment all available results completely.

This model was developed by M\'alek \& Rajagopal in \cite{MR} to simplify  the description of flows of complicated mixtures. Indeed, it was shown in \cite{MR} that it is a proper model for mixtures, where there is just one component that influences the mechanical properties of the fluid and this influence is then encoded into the Cauchy stress. The concentration of this particular component then fulfils the third equation in \eqref{problem}. This model with a proper choice of the form for $\S$ and $\q_c$ can be used for many complex materials as blood, synovial fluids, or in general biological fluids for example, see the thesis of Pust\v{e}jovsk\'a \cite{PT} or the corresponding paper \cite{HPMR}.

\subsection{Constitutive relations and boundary conditions}

The system \eqref{problem} must be equipped with the boundary conditions for $c$ and $\vv$ and completed by the constitutive relations for the Cauchy stress tensor $\S$ and the concentration flux $\q_c$. Since we are interested in the regularity theory, we simplify the paper by considering the spatial periodic conditions for $\vv$ and $c$. Although such a setting is nonphysical, it will be clear from the proof that we could obtain interior regularity result for realistic (e.g. Dirichlet condition) boundary conditions. Thus, in our setting the domain $\Omega$ will always be  a cube $[0,1]^d$ and all involved quantities are assumed to be periodic with respect to the cube $[0,1]^d$. In addition the emphasis will be given to two dimensional setting.

Concerning the form of the concentration flux, we shall assume that
\begin{equation}
\label{Ac1}
\q_c(c,\xi,\A):= \K(c, |\A|) \xi \qquad \textrm{ for all }(c,\xi,\A)\in \R\times \R^{d} \times \R^{d\times d},
\end{equation}
where $\K: \R^2 \to \R^{d\times d}$ is a continuous mapping fulfilling for all $(c,\xi,\A)\in \R\times \R^{d} \times \R^{d\times d}$
\be\label{concentration-flux-vector}
\begin{split}
|\K(c, |\A|)| &\leq K_2 ,\\
\K(c, |\A|) \xi  \cdot \xi &\geq K_1 |\xi|^2,
\end{split}
\ee
where $K_1$, $K_2$ are positive constants.

For the part of the Cauchy stress $\S$, we consider  that it is of the form
\begin{equation}
\label{AS1}
\S(c,\DD)=2\nu(c, |\DD|) \DD,
\end{equation}
where $\nu :\R^2 \to \R$ is a generalized kinematic viscosity. Furthermore, we require that there exist numbers $1<p^{-}\le 2 \le p^+ <\infty$ and a Lipschitz continuous function $p:\R \to [p^{-},p^+]$ such that for all $(c,\DD)\in \R \times \R^{d\times d}$ we have
\begin{align} \label{H1}
\left|\frac{\pa \S(c,\DD)}{\pa \DD}\right| &\leq
K_2(1+|\DD|)^{p(c)-2},\\
\label{H2}
\frac{\pa \S(c,\DD)}{\pa \DD}\cdot  (\B\otimes \B)&\geq K_1 (1+|\DD|)^{p(c)-2}|\B|^2,\\
\label{H3}
\left|\frac{\pa \S(c,\DD)}{\pa c} \right| &\leq K_2(1+|\DD|)^{p(c)-1}  \log(2+|\DD|),
\end{align}
where we used the notation $(\B\otimes \B)_{ijhk}=B_{ij}B_{hk}$.
As a direct consequence of assumptions \eqref{AS1}--\eqref{H3} one can also obtain that (after a possible change of constants $K_1$ and $K_2$) for all $(c,\DD)$ there holds (see \cite[Chapter 5]{MNRR} for detailed proof)
\begin{equation}
K_1 (1+|\DD|)^{p(c)-2}\le \nu(c,|\DD|) \le K_2(1+|\DD|)^{p(c)-2}.\label{H4}
\end{equation}

The typical model we have in mind is of the form
\be \label{model}
\S(c,\DD)= (1+\gamma(c)+ |\DD |^2)^{\frac{p(c)-2}{2}} \DD,
\ee
where $\gamma$ is a smooth bounded nonnegative function. Note that for proper functions $p$ and $\gamma$, the model \eqref{model} satisfies assumptions \eqref{H1}--\eqref{H3}.  It is important to notice that the presence of a function $p$ dependent on the concentration $c$ in the exponent is essential from the modelling point of view. It was well documented in \cite{PP,PT,HPMR} that the model \eqref{model} with properly chosen power function $p$ best fits the experimental data. Therefore, it is also our goal to cover this case in the paper.

\subsection{Notion of weak solution  and main result}
Since the constitutive relation involve the $p$-growth, which possibly depends on the concentration and therefore also on the spatial variable $x$, we need to introduce the corresponding Sobolev and Lebesgue spaces with variable exponent. Since, we also deal with the spatially periodic conditions, we fix the mean values of corresponding functions over the set $\Omega$ to zero. Hence, for a given function $p$ and periodic $c\in \mathcal{C}^{0,\alpha}(\mathbb{R}^2)$ for some $\alpha>0$, we introduce
\begin{align*}
\mathcal{C}^{\infty}_{\Per, \Div}&:= \{\vv\in \mathcal{C}^{\infty}(\R^d;\R^d): \, \int_{\Omega}\vv =0, \, \Div \vv=0, \, \vv \textrm{ is $\Omega$-periodic}\},\\
W^{1,p(c)}_{\Per, \Div}&:=\overline{\{\mathcal{C}^{\infty}_{\Per, \Div}\}}^{\|\cdot\|_{1,p(c)}}, \qquad W^{-1,p'(c)}_{\Per, \Div}:=(W^{1,p(c)}_{\Per, \Div})^*,
\end{align*}
where we use the equivalent norm $\|\vv\|_{1,p(c)}:= \|\nabla \vv\|_{p(c)}$. Since $p(c)$ is H\"{o}lder continuous and $1<p^-\le p^+<\infty$, these spaces are reflexive and separable. In addition, we also know that the the Korn inequality holds true, i.e., there exists a constant $C$ depending only on $p$ and $c$ such that
\begin{equation}
\|\vv\|_{1,p(c)}\le C\|\cD\|_{p(c)}.\label{Korn}
\end{equation}
Notice here that the constant $C$ depends on $c$ via the modulus of continuity of $c$. For details about the variable exponent function spaces we refer to \cite{DHHR}. Next, we also keep notation for standard Lebesgue and Sobolev spaces and the subscript $_{\Per}$ will denote that we consider $\Omega$-periodic functions having zero mean value over $\Omega$.

With this choice of function spaces we can define notion of a weak solution.
\begin{defi}\label{DF1}
Let be $\f \in W_{\Per, \Div}^{-1, (p_-)'}$ and $\g\in L^{q}_{\Per}(\O, \R^d)$ for some $q>d$. Let $\S$ satisfy \eqref{H1}--\eqref{H3}. We say that a couple $(c,\vv)$ is a weak solution to \eqref{problem}  if
\be\label{Def.space}
c\in W^{1,2}_{\Per}(\O) \cap \mathcal{C}^{0,\alpha}(\Omega), \qquad \vv \in W_{\Per, \Div}^{1, p(c)},
\ee
for some $\alpha >0$ and the system \eqref{problem} is fulfilled in the following sense
\begin{align}\label{wf1}
  \iO \Scv : \D \boldsymbol \psi \,dx  &=  \iO (\vv\otimes \vv) : \na \boldsymbol \psi\,dx+\langle \f, \boldsymbol \psi \rangle &&\forall \boldsymbol \psi \in \mathcal{C}^{\infty}_{\Per, \Div},\\
\label{wf2}
\iO \K(c, |\cD|) \na c \cdot \na \varphi \,dx &= \iO (c \vv+ \g) \cdot \na \varphi \,dx &&\forall \boldsymbol \varphi \in W_{\Per}^{1, 2}.
\end{align}
\end{defi}
In addition we have the following existence result.
\begin{tho}
Let all assumptions of Definition~\ref{DF1} be satisfied and
\begin{equation}\label{pmin}
p^->\frac{2d}{d+2}.
\end{equation}
Then there exists a weak solution provided that one of the following holds
\begin{itemize}
\item[i)] $p$ is independent of $c$,
\item[ii)] $p^->d/2$.
\end{itemize}
\end{tho}
\begin{proof}
We do not prove this result here since it is an easy modification of corresponding proofs in \cite{BuMaRa09,BP,BPP}, where the same system \eqref{problem} is treated but completed by Dirichlet boundary condition. We would like to mention that the case with $p$ being independent of $c$ is treated in \cite{BuMaRa09}. The case of general $p$ but with the restriction $p^{-}>3d/(d+2)$ is proved in \cite{BP} and the case when \eqref{pmin} holds is discussed in \cite{BPP}.
\end{proof}

\subsection{Main results}
We see that the existence analysis can be understood as a completed task but the regularity of the solution remains open. Our goal is to show the existence of a classical solution in case that $d=2$. The key result is the following.
\begin{tho}\label{maintheorem}
Let all assumptions of Definition~\ref{DF1} be satisfied with $d=2$. Assume in addition that $\f \in L^{2+2r}(\O; \R^2)$ with some $r>0$.
Then there exists a couple $(c,\vv)$ that is a weak solution in sense of Definition~\ref{DF1}, which fulfills in addition
\begin{equation}
\cD \in \mathcal{C}^{0,\alpha}_{\Per}(\O;\R^2).\label{calpha}
\end{equation}
In addition, if $p$, $\S$ and  $\K$ are smooth mappings and $\f$ and $\g$ are smooth vector--valued functions, then the constructed weak solution is smooth as well.
\end{tho}

The system \eqref{problem}  with the  constitutive  equations \eqref{Ac1} and \eqref{model} have been studied by many authors during  last decades for the  case $p$ constant or  the case when $p(x)$ is a given H\"{o}lder continuous function. However, we  are still far from saying that the theory is unified and satisfactory. While for two dimensional case, we know  that the velocity gradient is always H\"{o}lder continuous, see \cite{DER}, \cite{KMS} and \cite{KMSJ}, the results for three dimensional setting is indeed incomplete, see  \cite{B}, \cite{BKR} and \cite{Sin}. To our best knowledge, we have only partial regularity result, see \cite{ER}, or the  global result but only for small data, see \cite{CG} and \cite{CGC}. In any dimension there is a result of partial regularity, see \cite{AcMi}. It is remarkable, that this lack of regularity results, or more precisely, the analytical problems are coming from the fact that we have to deal with the nonlinearity depending on the symmetric gradient. The case when the nonlinearity is depending only on the full gradient, was successfully treated in \cite{CriMa} and \cite{CriMaNS} for constant $p$. The problem we  have to face in this paper is even more delicate. We do not know the variable exponent $p$ a priori but it is a part of the solution. Therefore, we have to develop a new technique that is capable to handle this problem. In view of the results for given $p$, we also first naturally focus only on two dimensional case, see Theorem \ref{maintheorem}. It is however notable, that even for more dimensional setting one can still derive the estimates for the second derivatives of the velocity field, which then lead to the partial regularity result as in \cite{ER}, which we will discuss in the forthcoming paper.
Finally we want to point out here, that the method used in the paper is based on the Hole-Filling technique of Widman (see \cite{W}) and proper Poincar\'{e} weighted inequalities, and is completely new in the setting of non-Newtonian fluids.

\section{Approximative problem, existence of its solution and uniform $W^{2,2}$-regularity}
We start the proof by defining an auxiliary approximation. Fixing arbitrary $A>1$, $c\in \R$ and $\B\in \R^{d\times d}$, we define
\be\label{tha}
\begin{split}
\tha(\B)&:= (2+\min\{A^2,|\B|^2\})^{\frac12},\\
\nu_A(c,\B)&:=\nu(c, \min\{A,|\B|\}).
\end{split}
\ee
Then, we introduce  the approximated stress tensor $\S^A$  by
\be \label{c33}
\S^A(c, \cD):= \nu_A(c,|\cD|)\cD
\ee
and the corresponding approximated problem
\be\label{problemA}
\begin{split}
\Div (\vv \otimes \vv)- \Div \S^A(c,\cD) &= -\na \pi + \f,\\
\Div \vv &=0,\\
\Div  {c \vv} - \Div \q_c(c, \na c, \cD)&= -\Div \g.
\end{split}
\ee
Next, we will prove the existence of a solution $(c,\vv)$ to \eqref{problemA} and show that $\vv\in \mathcal{C}^{1,\alpha}$ for some $\alpha$. Finally, our main goal will be to find a constant $A>1$ (typically sufficiently large) such that the solution of \eqref{problemA} satisfies
\begin{equation}
\|\cD\|_{\infty}\le A. \label{dreamA}
\end{equation}
Then evidently, $\vv$ is not only the solution to \eqref{problemA} but solves also the original problem \eqref{problem}. Consequently, we will get the existence of $\mathcal{C}^{1,\alpha}$ solution to the original problem.

The rest of this section is devoted to the construction of a solution to \eqref{problemA} and to the proof of  a~priori estimates that will be independent of the choice of parameter~$A$. Since the rigorous proof of the existence result for general dimension $d\ge 2$ was  established in \cite{BP,BPP} for Dirichlet boundary data, we mention here only the parts essential for our studies and skip the details.

To end this introductory part, we just recall basic properties of $\S^A$ which are direct consequences of assumptions \eqref{H1}--\eqref{H3} and also of the definition \eqref{c33}. Hence, for  all $c\in\R$ and any $\C,\B\in\R^{2\times2}$ we have (the interested reader can find the proof in  \cite{MNRR})
\begin{align}\label{c00}
\left|\frac{\pa \S_{ij}^A(c,\B)}{\pa \B_{kl}}\right| &\leq
C\tha^{p(c)-2}(\B),\\
\label{c11}
\frac{\pa \S_{ij}^A(c,\B)}{\pa \B_{kl}} \C_{ij}\C_{kl}&\geq \lambda \tha^{p(c)-2}(\B)|\C|^2,\\
\label{c22}
\left|\frac{\pa\S^A(c,\B)}{\pa c}\right| &\leq C \log(\tha(\B))\tha^{p(c)-2}(\B)|\B|,
\end{align}
where $\lambda$ and $C$ are positive constant that are independent of $A$.

\subsection{Galerkin approximation and the first a~priori estimates}

In this section we consider the problem \eqref{problemA} but for simplicity we avoid writing $\S^A$ and keep writing $\S$. Nevertheless, we derive estimates that will not depend on $A$. In case that some parts of the estimates are $A$-dependent, we clearly denote it in what follows. Also to simplify the notation, we use a symbol $C$ to denote a generic constant whose value can however change line to line.

Although the first part of the existence proof is almost identical to the procedure developed in \cite{BP,BPP} we recall the main steps here for the sake of clarity. First, we take  $\{ \w_i \}_{i=1}^{\infty}$ the basis of $W_{\Per, \Div}^{1, 2}$ composed of eigenfunctions of the Stokes operator such that $\iO \w_i\cdot\w_j\,dx=\delta_{ij}$. Due to the periodic boundary conditions, the basis consists of functions $\{ \w_i \}_{i=1}^{\infty}$ fulfilling for some positive  $\lambda_i$ (see also \cite{MNRR})
 \be\label{eigenfunctions}
-\Delta \w_i=\lambda_i \w_i \ \ \ \mbox{in} \ \O.
\ee
Similarly, we consider $\{ z_i \}_{i=1}^{\infty}$  a basis of $W_{\Per}^{1, 2}$ such that $\iO z_i z_j\,dx=\delta_{ij}$. Then for positive, fixed $n, m \in \N$, we look for a couple $(\vv^{n,m}, c^{n,m})$ given by
\be \label{def-Galerkin}
\vv^{n,m}:= \sum_{i=1}^n \alpha_i^{n,m}\w_i, \quad c^{n,m}:= \sum_{i=1}^m \beta_i^{n,m} z_i,
\ee
where $\boldsymbol\alpha^{n,m}$ and $\boldsymbol\beta^{n,m}$ solve the following system of algebraic equations (since $\f \in L^2$, we can replace the duality pairing appearing in \eqref{wf1} by the integral)
\begin{align}\label{approxsyst1}
&\iO -(\vv^{n,m}\otimes \vv^{n,m}) : \na \w_i\,dx +{\iO \Scvn : \D\w_i \,dx  =  \iO  \f \cdot \w_i \,dx}, \\
\intertext{for all  $i=1,\dots, n$,}
\label{approxsyst2}
&\iO \K(c^{n,m}, |\cD^{n,m}|)\na c^{n,m} \cdot \na z_j \,dx = \iO (c^{n,m} \vv^{n,m}+\g)  \cdot \na z_j \,dx
\end{align}
for all $j=1, \dots, m$.
The existence of a solution to \eqref{approxsyst1}--\eqref{approxsyst2} can be shown by the fixed point theorem and we refer the interested reader to \cite{BP} for details. In addition, following step by step \cite{BP} we can let $m\to \infty$ to obtain a solution $(\vv^n,c^n)\in (\mathcal{C}^{\infty}_{\Per},W^{1,2}_{\Per})$ to the following problem
\begin{align}\label{syst-n1}
&\iO  -(\vv^n\otimes \vv^n ) : \na \w_i\,dx+ {\iO \S^n : \D\w_i \,dx  =  \iO  \f \cdot \w_i \,dx}, &&\textrm{for all } i=1,\dots, n,\\
\label{syst-n2}
&\iO \K(c^{n}, |\cD^{n}|)\na c^{n} \cdot \na \varphi\,dx = \iO (c^n \vv^n+\g) \cdot \na \varphi \,dx,  &&\textrm{for all }\varphi \in W_{\Per}^{1,2}(\O),
\end{align}
where $\vv^n$ is given by
\be \label{def-Galerkinm}
\vv^{n}:= \sum_{i=1}^n \alpha_i^{n}\w_i.
\ee
 Here, we also  used the abbreviation  ${\bf S}^{n}:=\mathbf{S}(c^n, \cD^n)$. Next, we derive the first a~priori estimates. Multiplying the i-th equation in \eqref{syst-n1} by $\alpha_i^n$ and taking the sum over $i=1,\dots, n$, setting $\varphi:=c^n$ in \eqref{syst-n2}, and using integration by parts and the fact that $\Div \vv^n=0$, we get the following two identities
\be
\begin{split}\label{estap1}
\iO\S^A(c^n(x), \cD^n):\cD^n\,dx&= \iO  \f \cdot \w_i \,dx,\\
\iO \K(c^{n}, |\cD^{n}|)\na c^{n} \cdot \na c^n\,dx &= \iO \g \cdot \na c^n \,dx.
\end{split}
\ee
Hence, employing \eqref{c33} and the Sobolev embedding $W^{1,1}\hookrightarrow L^2$,  one gets from the first identity that there exist $c_1>0$ such that
$$
c_1\iO|\cD|^{p_-}\,dx-C\leq
\lambda\iO \tha^{p(c)-2}(\cD^n) |\cD^n|^2\,dx \leq \lambda\|\f\|_{2}\|\vv^n\|_{1,p^-}.
$$
Thus using the  Korn inequality and the assumptions on~$\f$, we observe
\be\label{DvL2}
\|\vv^n\|_{1,p^{-}}^{p^-}\leq C \iO \tha^{p(c)-2}(\cD^n) |\cD^n|^2+1\,dx \leq  C\left(\|\f\|_{2}^{(p^-)'}+1\right)\leq C.
\ee
Consequently, using the embedding theorem and the fact that $p^{-}>1$, we obtain that for some $\tilde{q}>2$ there holds
\be\label{DvL2vv}
\iO  |\vv^n|^{\tilde{q}}\, dx \leq  C.
\ee

Next, we focus on estimates for $c^n$ that follows from the second identity in \eqref{estap1}. First, using the assumption on $\K$, see~\eqref{concentration-flux-vector}, and the H\"{o}lder inequality, we deduce
$$
\iO |\na c^n|^2\,dx\leq C \iO |\g| |\nabla c^n|\, dx \le C\|\g\|_2 \|\na c^n\|_2
$$
and, from the Sobolev-Poincar\'e embedding, the Young inequality and the assumption on $\g$ it follows
\be \label{nabla-cn-2}
\|c^n\|_{1,2}^2\le C \iO |\na c^n|^2\,dx\leq  C\|\g\|_2^2<C.
\ee
We would like to emphasize at this place that  the estimates \eqref{DvL2} and \eqref{nabla-cn-2} are independent of $n\in\N$ and $A>1$, which will be used in what follows.

\subsection{Improvement of the integrability of $\na c^n$}
Through the paper, it is absolutely essential that the concentration will be H\"{o}lder continuous. Since we are interested in dimension two, we can obtain such H\"{o}lder continuity result by showing that $\nabla c^n \in L^q(\O)$ for some\footnote{This is a simple typically 2D alternative to the approach presented in \cite{BP}, where the H\"older continuity of $c^n$ is proved by the De Giorgi-Nash-Moser technique.} $q>2$. Such an improvement of the integrability  of the concentration gradient will be proven  by using the  reverse H\"older inequality.

To do so, we first denote
$$
\tilde{\bf g}:=c^n\vv^n +{\bf g}.
$$
Next, we define $q_0:=\min\{q, (\tilde{q}+2)/2\}$, where $\tilde{q}$ comes from \eqref{DvL2vv}. Then, by using the H\"{o}lder inequality, the embedding theorem and the a~priori estimate \eqref{nabla-cn-2}, we deduce
\begin{equation}\label{estg}
\|\tilde{\g}\|_{q_0} \le \|\g\|_q + \|c^n\vv^n\|_{\frac{\tilde{q}+2}{2}} \le C\left( \|\g\|_q + \|\vv^n\|_{\tilde{q}} \|c^n\|_{\frac{\tilde{q}(\tilde{q}+2)}{\tilde{q}-2}}\right)\le C.
\end{equation}

%Next, we proceed slightly formally and consider that \eqref{syst-n2} is fulfilled point-wisely in~$\R^2$, which is possible if we extend all functions periodically.
Next, for arbitrary $x_0\in \R^2$ and $R\in (0,1)$ we test the equation by $\varphi =(c^n-c_R^n)\eta^2$, where we denote the mean value over the ball $B(x_0, R)$ with the subscript $_R$. The function $\eta$ is a cut-off function, i.e., $\eta \in \mathcal{C}_0^\infty(B(x_0,R))$, $0\leq \eta\leq 1$, $\eta\equiv 1$ in $B(x_0, {R}/{2})$ and $|\na \eta| \leq {k}/{R}$, extended periodically with respect to  $\Omega$. Moreover we extend to $\mathbb{R}^2$ periodically with respect to $\O$ all functions and obtain that \eqref{syst-n2}  is fulfilled in $\R^2$.
Then employing $\varphi$ as test function we get that (using also the definition of $\tilde{g}$)
\be\label{gradient-cn}
\begin{split}
\int_{\R^2} \K(c^n, |\cD^n|)|\na c^n|^2 \eta^2 \,dx + \int_{\R^2} \K(c^n, |\cD^n|)\na c^n\cdot \na (\eta^2) (c^n-c_R^n)\,dx \\
= \int_{\R^2} \tilde{\g} \cdot \na (\eta^2)(c^n-c_R^n)\,dx +
\int_{\R^2} \tilde{\g}\cdot\na c^n\eta^2\,dx.
\end{split}
\ee
Employing the properties of $\eta$, the assumption \eqref{concentration-flux-vector} and the Young inequality, the identity \eqref{gradient-cn} becomes
\be\begin{split}\label{p:1}
\iRt |\na c^n|^2\,dx &\leq C \iR \frac{|c^n-c_R^n|^2}{R^2}\,dx + C\iR |\tilde{\g}|^2\, dx.
\end{split}\ee
Using the Sobolev-Poincar\'e inequality, we can estimate the first term on the right hand side as (here $\imR$ denotes the mean value integral)
$$
\iR\frac{|c^n-c^n_R|^2}{R^2}\, dx
\leq CR^2\left(\imR|\na c^n| \, dx \right)^{2}.
$$
Hence, substituting this inequality into \eqref{p:1} and dividing both sides by $R^2$, we obtain
 \be\label{propGiaquinta}
\imRt |\na c^n|^2\, dx  \leq C \left( \imR |\na c^n| \,dx\right)^2
+C \imR |\tilde{\g}|^2\,dx.
\ee
Finally, recalling \eqref{estg}, we have $\tilde{\g}\in L^{q_0}$ with $q_0>2$. Thus,  we can employ the reverse H\"{o}lder inequality, see in \cite{GM} or in \cite[Proposition 1.1 (on page 122)]{Gbook} and conclude that
\be\label{summability-nabla-c}
|\na c^n| \in L_{{\rm loc}}^{2+\delta}(\O)
\ee
for certain $\delta>0$ that depends only on $q_0$ and $C$ and we have the estimate
\be
\bigg(\imRt \! \!|\na c^n|^{2+\delta}\,dx\bigg)^{\frac{1}{2+\delta}} \! \! \!\leq C\bigg( \imR \! \!\!\!|\na c^n|^2\,dx\bigg)^{\frac{1}{2}}+ C\bigg( \imR \!\!\!\!|\tilde{\g}|^{2+\delta}\,dx\bigg)^{\frac{1}{2+\delta}}.
\ee
Consequently, it follows  from \eqref{DvL2} and \eqref{nabla-cn-2} and the embedding theorem that
\be \label{cl2d}
\|c^n\|_{\mathcal{C}^{0,\frac{\delta}{2+\delta}}(\R^2)}+\|\na c^n\|_{L^{2+\delta}(\Omega)} \leq  C
\ee
where $\delta>0$ and $C>0$ are constants independent of $A$ and $n$.

\subsection{Global uniform $W^{2,2}$ estimates for $\vv^n$}\label{SS2.3}

In this section we derive the main starting estimate that will be uniform with respect to $A$ and $n$ and will play the crucial role for deriving the H\"{o}lder continuity of the velocity gradient.  It is important since the quality of this estimate determines how strict are our assumption on $p^-$ and $p^+$. We still work with the approximation $\{\mathbf{v}^n, c^n\}$ fulfilling \eqref{syst-n1} and \eqref{syst-n2} and suppress dependence on $n$ and $A$. However, the constant $\lambda$ is fixed by \eqref{c11} and the constant $C$ may vary line to line but will not depend on $n$ or $A$, but can depend on data only.

Multiplying the $i$-th equation in \eqref{syst-n1} by $\lambda_i\alpha_i^{n}$ and taking the sum over $i=1,\dots,n$, which is nothing else than testing by $-\Delta \vv$  in virtue of \eqref{eigenfunctions}, and using the fact that in dimension two it holds (using integration by parts)
$$
\iO (\vv\otimes\vv) : \na \Delta \vv\,dx = \sum_{i,j,k=1}^2 \iO \partial_{x_k}v_j \partial_{x_j}v_i \partial_{x_k} v_i\,dx =0,$$ for $\vv\in C^\infty_{per}(\Omega)$ we get
\be
-\iO \S : \D(\Delta\vv) \,dx  =  -\iO \f \cdot\Delta \vv\, dx.
\ee
After integration by parts we arrive at
\be\label{sumFr}
\iO  \partial_{x_k} [\S (c(x), \cD) \big]: \partial_{x_k} (\cD) \,dx=-\iO \f \cdot\Delta \vv\, dx,
\ee
where we used the Einstein summation convention. First, we focus on the term on the left hand side.
We apply the derivative to the corresponding term and denote
$$
\begin{aligned}
&\iO  \partial_{x_k} [\S (c(x), \cD) \big]: \partial_{x_k} (\cD) \,dx\\
&\qquad= \iO (\partial_{c}\S :\partial_{x_k} (\cD))\partial_{x_k}c\,dx
+\iO \partial_D \S:(\partial_{x_k} (\cD)\otimes\partial_{x_k} (\cD)) \,dx
=:F_1+F_2.
\end{aligned}
$$
The good term $F_2$ can be estimated by \eqref{c11}
as
\be\label{F2}
F_2 \geq \lambda\iO \tha^{p(c)-2}(\cD) |\na \cD|^2\, dx.
\ee

So we focus on the worse term $F_1$. By virtue of  \eqref{c22}, it holds
$$
|F_1| \leq C\iO |\nabla c| \log(\tha(\cD))\tha^{p(c)-2}(\cD)|\cD||\na \cD|\,dx.
$$
Now we apply the  H\"older inequality with the exponent $(2+\delta)$, where $\delta$  appears in \eqref{cl2d}, the exponent $2$ and  exponents $\alpha>2$ and $\beta>2$, (which will be from now fixed, depending thus only on already fixed $\delta>0$)  such that
$$
1=\frac{1}{2+\delta}+\frac 1\alpha+\frac1\beta+\frac12
$$
to get
\be\label{F1}
|F_1| \leq C\|\nabla c\|_{2+\delta}\| \log(\tha(\cD))\|_{\alpha}\|\tha^{\frac{p(c)-2}2}(\cD)|\cD|\|_{\beta}\|\|\tha^{\frac{p(c)-2}2}(\cD)\na \cD|\|_2.
\ee
The first term can be simply estimated by \eqref{cl2d}. For the second term
we use the fact that $\log(s)\leq C s^{p^-/\alpha}$ for $s>1$ and due to the a~priori bound \eqref{DvL2}, we have  that
$$
\| \log(\tha(\cD))\|_{\alpha}\leq C\|1+\cD\|_{p^-}^{p^-/\alpha}\leq C.
$$
Thus, \eqref{F1} reduces to
\be\label{F111}
|F_1| \leq C\|\tha^{\frac{p(c)-2}2}(\cD)|\cD|\|_{\beta}\|\|\tha^{\frac{p(c)-2}2}(\cD)\na \cD|\|_2.
\ee
Finally, abbreviating $\eta:=\tha(\cD)^{(p(c)-2)/2}\cD$, we also have by using \eqref{DvL2} and that $p^-\leq2$
$$
\|\eta\|_{p^-}\leq C,\quad |\na\eta|\leq C\tha^{\frac{p(c)-2}2}(\cD)|\na \cD|+ C|\nabla c| \log (\tha (\cD))\tha^{\frac{p(c)-2}2}(\cD)|\cD|
$$
and, using again the H\"{o}lder inequality and the same procedure as above, we have
\be\label{eta-gradeta}
\|\eta\|_{1,2}\leq C\left(1+\|\tha^{\frac{p(c)-2}2}(\cD)|\na \cD|\|_2+ \|\eta\|_{\beta}\right).
\ee
Consequently, for given $\beta$ we can find  $\sigma\in(0,1)$ such that after using the interpolation theorem, we have with the help of \eqref{eta-gradeta} that
$$
\begin{aligned}
\|\eta\|_\beta\leq C(\beta)\|\eta\|_{p^{-}}^{\sigma}\|\eta\|_{1,2}^{1-\sigma} \le C\left(1 +\|\tha^{\frac{p(c)-2}2}(\cD)|\na \cD|\|^{1-\sigma}_2+ \|\eta\|^{1-\sigma}_{\beta}\right).
\end{aligned}
$$
Since $\sigma>0$ we can use the Young inequality to move the last term to the left hand side to obtain the final estimate
\be\label{eta_beta}
\begin{aligned}
\|\eta\|_\beta\leq C+ C\|\tha^{\frac{p(c)-2}2}(\cD)|\na \cD|\|^{1-\sigma}_2.
\end{aligned}
\ee
Combining \eqref{F111} and \eqref{eta_beta} and the Young inequality, we deduce
\be\label{F1-2}
\begin{aligned}
|F_1| &\leq C +C \|\tha^{\frac{p(c)-2}2}(\cD)\na \cD|\|_2^{2-\sigma} \le C +\frac{\lambda}{4}\|\tha^{\frac{p(c)-2}2}(\cD)\na \cD|\|^2_2.
\end{aligned}
\ee

It remains to estimate the term on the right hand side of \eqref{sumFr}. Using the H\"{o}lder and the Young inequalities, the assumption on $\f$ and the uniform estimate \eqref{DvL2}, we find
\begin{equation} \label{vloz}
\begin{aligned}
\left|\iO \f \cdot \Delta \vv\, dx \right| &\le C\iO |\f| \, \tha^{\frac{2-p(c)}{2}}(\cD) \,\tha^{\frac{p(c)-2}2}(\cD)|\na \cD|\, dx \\
&\le  \|\f\|_{2+2r} \|\tha^{\frac{p(c)}{2}\frac{2-p(c)}{p(c)}}(\cD)\|_{\frac{2(1+r)}{r}}\| \tha^{\frac{p(c)-2}2}(\cD)\na \cD|\|_2\\
&\le C\|\tha^{\frac{p(c)}{2}}(\cD)\|_{\frac{2(1+r)}{r}\frac{2-p^-}{p^-}}^{\frac{2(2-p^-)}{p^-}} + \frac{c}{4}\| \tha^{\frac{p(c)-2}2}(\cD)\na \cD|\|_2^2.
\end{aligned}
\end{equation}
Hence, inserting estimates \eqref{vloz}, \eqref{F1-2} and \eqref{F2} into \eqref{sumFr}, we obtain
\begin{equation}\label{est:w12H}
\|\tha^{\frac{p(c)-2}2}(\cD)\na \cD\|_2\leq C\left(1+\|\tha^{\frac{p(c)}{2}}(\cD)\|_{\frac{2(1+r)}{r}\frac{2-p^-}{p^-}}^{\frac{2-p^-}{p^-}}\right)
\end{equation}
with the constant $C$ independent of $A$ (and also of $n$). In addition, it also follows from \eqref{eta-gradeta}, \eqref{eta_beta} and the definition of $\eta$ that
\begin{equation} \label{n12H}
\|\tha^{(p(c)-2)/2}(\cD)\cD\|_{1,2} \le C\left(1+\|\tha^{\frac{p(c)}{2}}(\cD)\|_{\frac{2(1+r)}{r}\frac{2-p^-}{p^-}}^{\frac{2-p^-}{p^-}}\right).
\end{equation}
Consequently, we can use the embedding theorem to obtain that for any $\beta\in (1,\infty)$ there holds
\begin{equation}\label{beta-normA}
\|\tha^{\frac{p(c)}{2}}(\cD)\|_{\beta}\le \|\tha^{(p(c)-2)/2}(\cD)\cD\|_{\beta}\le C(\beta)\|\tha^{(p(c)-2)/2}(\cD)\cD\|_{1,2}.
\end{equation}
Thus, setting $\beta:=\frac{2(1+r)}{r}\frac{2-p^-}{p^-}$ in \eqref{beta-normA}, inserting this inequality into \eqref{n12H}, using the fact that $(2-p^{-})/p^- <1$ and the Young inequality, we deduce that
\begin{equation} \label{n12}
\|\tha^{(p(c)-2)/2}(\cD)\cD\|_{1,2} \le C.
\end{equation}
Then, it directly follows from \eqref{n12H} and \eqref{beta-normA} that
\begin{equation}\label{est:w12}
\|\tha^{\frac{p(c)-2}2}(\cD)\na \cD\|_2\leq C
\end{equation}
and
\begin{equation}\label{beta-norm}
\|\tha^{\frac{p(c)}{2}}(\cD)\|_{\beta}\le \|\tha^{(p(c)-2)/2}(\cD)\cD\|_{\beta}\le C(\beta)
\end{equation}
with the constant $C$ independent of $n$ and $A$. In addition, it also follows from \eqref{beta-norm}, the definition of $\tha$ and the fact that $p^->1$ that \begin{equation}\label{beta-norm2}
\|\cD\|_{\overline{\beta}}\le C(\overline{\beta}), \qquad \|\tha^{\frac{p}{2}}(\cD)\|_{1,2}\le C,
\end{equation}
for any $\overline{\beta}\in (1,\infty)$.

\subsection{Limit $n\to \infty$}
Thanks to the regularity estimate \eqref{beta-norm2} and recalling also \eqref{cl2d}, we can use the compact embedding and the monotone operator theory to let $n\to \infty$ in \eqref{syst-n1}--\eqref{syst-n2} to deduce the existence of a couple $(c^A,\vv^A)$ fulfilling
\begin{align}
  \label{syst-w1}
&\iO-(\vv^A\otimes \vv^A) : \na \w + \S^A(\D\vv^A) : \D\w \,dx  =  \iO\f \cdot \w \, dx  &&\textrm{for all }\w\in W^{1,2}_{\Per, \Div},\\
\label{syst-w2}
&\iO \K(c^A, |\cD^A|)\na c^{A} \cdot \na \varphi\,dx = \iO (c^A \vv^A+\g) \cdot \na \varphi \,dx &&\textrm{for all } \varphi \in W_{per}^{1,2}(\O).
\end{align}
Moreover, thanks to uniform estimates \eqref{cl2d} and \eqref{n12}, and due to the definition of $\theta_A (\cD)$, we know that
\begin{equation}\label{reg:nonunif}
  \vv^A\in W^{2,2}_{loc}(\mathbb{R}^d; \mathbb{R}^d),\quad c^A\in \mathcal{C}^{0,\frac \delta{2+\delta}}(\R^2)\cap W^{1,2+\delta}_{\Per}
\end{equation}
for a suitable $\delta>0$, and they fulfill the estimates \eqref{cl2d}, \eqref{est:w12} uniformly with respect to $A>1$.

\section{Proof of the main theorem}
This section is devoted to the proof of the main theorem. First, we recall and improve some standard results in the regularity theory. Then we localize the estimate arising in Subsection~\ref{SS2.3}. Next, we use the hole filling technique to show the sharp $\mathcal{C}^{\alpha}$ estimates and we also trace the precise dependence on the parameter $A$. Finally, we combine such result with the uniform bound \eqref{n12} to show an estimate which is independent of $A$.

To end this introductory part, we recall some standard notation that will be used for localization. For any $x\in \R^2$ and $R>0$ the symbol $B_{R}(x)$ denotes the ball centered at~$x$ and radius~$R$. Similarly, we denote the annulus $A_R(x):=B_{2R}(x)\setminus B_{R}(x)$. Often, in case it is clear from context, we will omit writing the center $x$. Further, we introduce a notation for a function with zero average. For a function $f$ defined on a measurable set $U$ we define
$$
\meanless{f}{U}:=f-(f)_{U},
$$
where
$$
(f)_U:=\fint_U f\, dx = \frac{1}{|U|}\int_U f\, dx.
$$
For $U=A_R$ we shorten the notation to $\meanless fR$.

\subsection{Auxiliary estimate}
We recall here the standard hole filling lemma (see \cite{W}), where we however sharply trace the dependence on all constants.
\begin{lem}[Hole Filling Lemma]\label{HoleFillLem}
Let $g\in L^1_{loc}(\mathbb{R}^2)$ and $\alpha, \beta$ and  $\nu$ be positive constants. Assume that for all $0<R\le R_0\le 1$ the following inequality holds true
\be\label{hyp-lem}
 \int_{B_{R}} |g|\,dx \leq \alpha \iRR  |g|\,dx + \beta R^\nu.
\ee
If we define
\be \label{dfmu}
\mu:=\min \left\{\frac{\nu}{2}, \log_2 \left(\frac{1+\alpha}{\alpha}\right)\right\}
\ee
then for all $R\in (0,R_0)$ there holds
\be\label{hole-result}
 \int_{B_{R}} |g|\,dx \leq R^{\mu} \left( 2^{\nu}\int_{B_{R_0}} \frac{|g|}{R_0^{\mu}}\,dx+ \frac{\beta}{2^{\frac{\nu}{2}}-1} \right).
\ee
\end{lem}

\begin{proof}
We add  $\alpha \displaystyle \int_{B_{R}} |g|\,dx$ to both sides of \eqref{hyp-lem} and after division the result by $(1+\alpha)$,  we get
\be\label{iter0}
 \int_{B_{R}} |g|\,dx \leq \frac{\alpha}{1+\alpha} \int_{B_{2R}} |g|\,dx+ \frac{\beta}{1+\alpha} R^\nu.
\ee
Next, we add to both sides $\varepsilon^{-1} R^{\nu}\beta/(\alpha+1)$ with some $\varepsilon>0$, that will be specified later, and divide the result by $R^{\mu}$ to obtain
\be\label{iter1}
\begin{split}
 \int_{B_{R}} \frac{|g|}{R^{\mu}} \,dx+ \frac{\beta \varepsilon^{-1}}{1+\alpha} R^{\nu-\mu} &\leq \frac{\alpha}{1+\alpha} \int_{B_{2R}} \frac{|g|}{R^{\mu}}\,dx+ \frac{(1+\varepsilon^{-1})\beta}{1+\alpha} R^{\nu-\mu}\\
 &=\frac{\alpha 2^{\mu}}{1+\alpha} \int_{B_{2R}} \frac{|g|}{(2R)^{\mu}}\,dx+ 2^{\mu-\nu}(\varepsilon+1)\frac{\varepsilon^{-1}\beta}{1+\alpha} (2R)^{\nu-\mu}.
\end{split}
\ee
Hence, denoting
$$
\eta(R):=\int_{B_{R}} \frac{|g|}{R^{\mu}}\,dx+ \frac{\varepsilon^{-1}\beta}{1+\alpha} R^{\nu-\mu},
$$
we obtain from \eqref{iter1} that
\be \label{iter4}
\eta(R)\leq \eta(2R) \max\left\{\frac{\alpha 2^{\mu}}{1+\alpha},2^{\mu-\nu}(\varepsilon+1)\right\}.
\ee
Finally, setting $\varepsilon:= 2^{\frac{\nu}{2}}-1>0$, we can use the definition of $\mu$, see \eqref{dfmu}, to deduce that
$$
\max\left\{\frac{\alpha 2^{\mu}}{1+\alpha},2^{\mu-\nu}(\varepsilon+1)\right\}\le 1.
$$
Consequently, \eqref{iter4} reduces to
\be \label{iter5}
\eta(R)\leq \eta(2R).
\ee
Thus, for any $R>0$ we can find $m\in \mathbb{N}$ such that $2^mR \in (R_0/2 ,R_0)$ and iterating \eqref{iter5}, we see that (using also the fact that $\mu\le \nu$ and $R_0\le 1$)
\be
\begin{split}\label{iter6}
\eta(R)\le \eta(2^mR)&=\int_{B_{2^mR}} \frac{|g|}{(2^m R)^{\mu}}\,dx+ \frac{\varepsilon^{-1}\beta}{1+\alpha} (2^mR)^{\nu-\mu}\\
&\le 2^{\nu}\int_{B_{R_0}} \frac{|g|}{(R_0)^{\mu}}\,dx+ \frac{\beta}{2^{\frac{\nu}{2}}-1}.
\end{split}
\ee
Using the definition of $\eta$, we see that \eqref{hole-result} easily follows.
\end{proof}

\subsection{Local estimates of second gradient.}
In this section we closely follow  the procedure developed in Section~\ref{SS2.3} but we will focus on localized estimate.
\begin{lem}
Let $\vv$ be the solution determined in \eqref{syst-w1}-\eqref{reg:nonunif}, then  for any $x_0\in \R^2$ and any $R\in (0,1)$ it holds
\be \label{hole-start}
\int_{B_{R}(x_0)}\!\!\!\!\!\!\!\!\!\!\theta_A^{p(c)-2}(\cD) \left| \na \cD\right|^2 \,dx \leq CR^{\nu}+C\int_{A_R(x_0)}\!\!\!\!\!\!\!\!\!\!\tha^{p(c)-2}(\cD) \frac{\left|\meanless{\na \vv}{R} \right|^2}{R^2}\,dx,
\ee
where the positive constants $C$ and $\nu$ are independent of $A$.
\end{lem}
\begin{proof}
We omit writing $x_0$ in what follows. We also proceed here formally assuming that all test functions are smooth\footnote{Such a procedure can be however for our problem justified rigorously} enough.  Let $\tau_R \in \mathcal{C}^1_0(B_{2R})$ be a cut-off function such that
$\tau_R=1$ on $B_R$ and $|\nabla\tau_R|\leq \frac{C}{R}$ on $B_{2R}\setminus B_R$ and consider $\w:=- \diver (\meanless{\nabla \vv-\nabla \vv^T}{R} \tau_R^2)$ as a test function in \eqref{syst-w1}. Note that $\w$ is divergence free, but just belongs to $L^2$ so we put the derivative on $\S$ and $\vv\otimes\vv$. Next, we repeat almost step by step the computations in Section~\ref{SS2.3} with the necessary changes due to the localization. Hence, we evaluate and estimate all terms arising by using $\w$ as a test function and for these estimates we repeatedly use \eqref{c00}--\eqref{c22}, integration by parts, the H\"older, the Poincar\'{e} and the Young inequalities. First, for the convective term, we have (using integration by parts twice and the Young inequality)
\begin{align*}
&\left|\int_{\R^2}(\vv \otimes \vv):\nabla \w\, dx\right| = \left| \int_{\R^2}  \na \vv : (\vv \otimes  \w)\,dx\right|\\
&=\left|\int_{\R^2}\na \vv : (\vv \otimes\diver (\meanless{\nabla \vv-\nabla \vv^T}{R} \tau_R^2))\, dx \right|\\
&\leq  \int_{B_{2R}} (|\nabla^2\vv||\vv| + |\nabla \vv|^2) |(\meanless{\nabla \vv-\nabla \vv^T}{R}|\tau_R^2\,dx\\
& \leq \frac{\lambda}{4} \int_{\R^2} \tha^{p(c)-2}(\cD)|\nabla \cD|^2\tau_R^2\, dx+ C\int_{B_{2R}}|\vv|^2\theta_A^{2-p(c)}|\meanless{\nabla \vv-\nabla \vv^T}{R}|^2 + |\nabla \vv|^3\, dx   \\
& \leq \frac{\lambda}{4} \int_{\R^2} \tha^{p(c)-2}(\cD)|\nabla \cD|^2\tau_R^2\, dx + C\int_{B_{2R}}|\vv|^6\theta_A^{3(2-p(c))}+ |\nabla \vv|^3\, dx.
\end{align*}
Finally, for the second integral, we can use estimates \eqref{beta-norm}--\eqref{beta-norm2}, the embedding theorem  and the Korn inequality to obtain
\begin{align*}
&\left|\int_{\R^2}(\vv \otimes \vv):\nabla \w\, dx\right| \leq \frac{\lambda}{4} \int_{\R^2}\!\!\! \tha(\cD)^{p(c)-2}|\nabla \cD|^2\tau_R^2\, dx+CR \| |\vv|^6\theta_A^{3(2-p(c))}+|\nabla \vv|^3\|_2 \\
&\leq  \frac{\lambda}{4} \int_{\R^2} \tha(\cD)^{p(c)-2}|\nabla \cD|^2\tau_R^2\, dx+CR
\end{align*}
For the term with $\f$, we again use the H\"{o}lder, the Poincar\'{e} and the Young inequality, and recalling that due to the assumptions we know that $\f\in L^{2(1+r)}(\Omega; \R^2)$ and the uniform estimates~\eqref{est:w12} and \eqref{beta-norm2}, we have
\begin{align*}
&\left| \int_{\R^2} \f \cdot \w \,dx \right| \leq \int_{\R^2} |\f||\Delta \vv|\tau^2_R\, dx  + C\iRR |\f| \frac{\left|\meanless{\na \vv - \na \vv^T}{R} \right|}{R}\,dx\\
&\leq C\int_{B_{2R}} |\f|\tha^{\frac{2-p(c)}{2}} \tha^{\frac{p(c)-2}{2}} |\nabla \cD|\, dx  \\
&\qquad + C\left(\iRR |\f|^{2+2r}\, dx\right)^{\frac{1}{2+2r}}\left(\iRR |\nabla \cD|^{\frac{2+2r}{1+2r}}\, dx  \right)^{\frac{1+2r}{2+2r}}\\
&\leq C\|\f\|_{2+2r} \|\tha^{\frac{p(c)-2}{2}} \nabla \cD\|_2 \left(\int_{B_{2R}} \tha^{\frac{2-p(c)}{2}\frac{2(1+r)}{r}} \, dx\right)^{\frac{r}{2(1+r)}} \\
 &\qquad + C\|\f\|_{2+2r}\left(\iRR \!\!\!\tha^{(2-p(c))\frac{1+r}{1+2r}}(\tha^{p(c)-2}|\nabla \cD|^2)^{\frac{1+r}{1+2r}}\, dx  \right)^{\frac{1+2r}{2+2r}}\\
 &\leq C\left(\int_{B_{2R}} \tha^{\frac{(2-p(c))(1+r)}{r}} \, dx\right)^{\frac{r}{2(1+r)}}\le \|\tha^{\frac{(2-p(c))(1+r)}{r}}\|_2^{\frac{r}{2(1+r)}}R^{\frac{r}{2(1+r)}}\le CR^{\frac{r}{2(1+r)}}.
\end{align*}
The last term we need to estimate is the one with $\S$. We use the inequalities \eqref{c00}--\eqref{c22}, integration by parts, the H\"{o}lder inequality, the Poincar\'{e} inequality  and the Young inequality and also the uniform estimates \eqref{cl2d} and \eqref{est:w12}--\eqref{beta-norm2} to obtain (we proceed here without details since the very similar procedure was already used in Section~\ref{SS2.3})
\begin{align*}
-\int_{\R^2} &\diver \S \cdot \w \,dx=\int_{\R^2} \diver \S \cdot  \diver (\meanless{\nabla \vv-\nabla \vv^T}{R}\tau^2_R)  \,dx\\
&=\int_{\R^2} \diver \S \cdot  \diver (\meanless{\nabla \vv}{R}\tau_R^2) -(\diver \S \otimes \nabla \tau^2_R) :  (\meanless{\nabla \vv^T}{R})  \,dx \\
&\ge \int_{\R^2} \nabla \S :  \nabla^2 \vv \tau_R^2\, dx -R^{-1}C\int_{A_R}|\nabla \S| |\meanless{\nabla \vv}{R}|\tau_R  \,dx \\
&\ge c\int_{\R^2} \tha^{p(c)-2} \left| \na \cD \right|^2 \tau_R^2\,dx- C\int_{\R^2} \log(\tha(\cD))\tha^{p(c)-2}(\cD)|\cD||\nabla \cD||\nabla c|\tau_R^2\, dx\\ &\qquad -CR^{-1}\int_{A_R} \tha^{\frac{p(c)-2}{2}} |\nabla \cD|\tau_R  \tha^{\frac{p(c)-2}{2}}\left|\meanless{\na \vv }{R} \right|\,dx  \\ 	
&\qquad -C R^{-1}\int_{A_R} \log(\tha(\cD))\tha^{p(c)-2}(\cD)|\cD||\meanless{\nabla \vv}{R}||\nabla c|\, dx \\
&\ge \frac{3c}{4}\int_{\R^2} \tha^{p(c)-2} \left| \na \cD \right|^2 \tau_R^2\,dx- C\int_{B_{2R}} \log^2(\tha(\cD))\tha^{p(c)-2}(\cD)|\cD|^2|\nabla c|^2\, dx\\ &\qquad -CR^{-2}\int_{A_R} \tha^{p(c)-2} \left|\meanless{\na \vv }{R} \right|^2\,dx  \\ 	
&\qquad -C R^{-1}\int_{B_{2R}} \log(\tha(\cD))\tha^{p(c)-2}(\cD)|\cD||\meanless{\nabla \vv}{R}||\nabla c|\, dx.
\end{align*}
Next, we can estimate the second and the last integral as follows (using \eqref{cl2d} and \eqref{beta-norm2})
$$
\begin{aligned}
\int_{B_{2R}} &\log^2(\tha(\cD))\tha^{p(c)-2}(\cD)|\cD|^2|\nabla c|^2\, dx\\
 &\le C \|\nabla c\|_{L^{2+\delta}(B_{2R})}^2 \|(1+|\cD|)^{p^++1}\|_{L^{\frac{2+\delta}{\delta}}(B_{2R})}\\
&\le \|\nabla c\|_{L^{2+\delta}(B_{2R})}^2 \|(1+|\cD|)^{p^++1}\|_{L^{\frac{2(2+\delta)}{\delta}}(B_{2R})}R^{\frac{\delta}{\delta+2}}\le CR^{\frac{\delta}{\delta+2}}.
\end{aligned}
$$
The last integral is estimated similarly (using Korn inequality)
$$
\begin{aligned}
R^{-1}&\int_{B_{2R}} \log(\tha(\cD))\tha^{p(c)-2}(\cD)|\cD||\meanless{\nabla \vv}{R}||\nabla c|\, dx \\
&\le CR^{-1}\|\nabla c\|_{L^{2+\delta}(B_{2R})} \|(1+|\cD|)^{p^+ +1}\|_{L^{\frac{2+\delta}{1+\delta}}(B_{2R})}\\
&\le CR^{-1}\|\nabla c\|_{L^{2+\delta}(B_{2R})} \|(1+|\cD|)^{p^+ +1}\|_{L^{\frac{4(2+\delta)}{\delta}}(B_{2R})}R^{\frac{3\delta+4}{2(2+\delta)}}\le C R^{\frac{\delta}{2(2+\delta)}}.
\end{aligned}
$$
Thus, defining
$$
\nu:=\min \left\{\frac{\delta}{2(2+\delta)}, \frac{r}{2(1+r)}\right\}
$$
and summarizing all above inequalities, we finally deduce \eqref{hole-start}.
\end{proof}

%\begin{lem}
%Let $\vv$ be the solution determined in \eqref{syst-w1}-\eqref{reg:nonunif}, then for any $\beta\geq 1$
%\be
%\|\cD\|_\beta \leq K,
%\ee
%for any $A>1$.
%\end{lem}
%
%\Pr Let us recall that  $\eta=\tha(\cD)^{(p(c)-2)/2}\cD$ fulfills \eqref{eta-gradeta} and from \eqref{eta_beta} and \eqref{est:w12} for any $\beta\geq1$ one gets
%\be
%\|\eta\|_\beta\leq K.
%\ee
%It is easily seen that if $p(c)\geq2$ then $|\eta|\geq |\cD|$ and thus the assertion follows. Assume now $p(c)<2$, then
%\be\ba{l}\displ
%|\eta|= \tha(\cD)^{(p(c)-2)/2}\left(1+|\cD|^2\right)^{\frac{p(c)-2}{4}}\left(1+|\cD|^2\right)^{\frac{2-p(c)}{4}}|\cD|\\ \displ
%\hfill\geq \left(1+|\cD|^2\right)^{\frac{p(c)-2}{4}}|\cD|.
%\ea
%\ee
%Splitting into the subcases $|\cD|\leq1$ and $|\cD|>1$ we get
%\be
%|\eta|\geq K |\cD|,
%\ee
%and thus the assertion.
%\\\qed

\subsection{Covering by proper balls}\label{covering}
We use the uniform estimate \eqref{cl2d}, which is independent of $A$,  to specify a proper covering of $\Omega$. We take arbitrary positive $\varepsilon_0 \le 1/10$ and find $R_0\le 1$ such that for all $x,y\in \R^2$ fulfilling $|x-y|\le 8R_0$, we have $|p(x)-p(y)|\le \varepsilon_0$. Then we can find a finite number of points $\{x_i\}_{i=1}^N$ such that $\cup_{i=1}^N B_{R_0/2}(x_i)$ is covering of $\Omega$. In addition, we define
$$
\begin{aligned}
p_i^{-}:= \inf_{x\in B_{8R_0}(x_i)} p(c(x)), \qquad p_i^{+}:= \sup_{x\in B_{8R_0}(x_i)} p(c(x)),\\
p_i^{-}(R):= \inf_{x\in B_{8R}(x_i)} p(c(x)), \qquad p_i^{+}(R):= \sup_{x\in B_{8R}(x_i)} p(c(x)).
\end{aligned}
$$
Finally, we also introduce constants related to balls $B_{8R_0}(x_i)$, which give the upper and the lower estimate for $\tha$
\be \label{cA}
\begin{split}
  c_0^i(A)&:= (1+|A|^2)^{\frac{\max\{2,p^+_i\}-2}2},\\
  c_1^i(A)&:=(1+|A|^2)^{\frac{\min\{2,p^-_i\}-2}2}.
\end{split}
\ee
Indeed, it directly follows from the definition \eqref{tha} that
\be \label{cAA}
\begin{aligned}
  c_0^i(A)&\ge \lambda\tha (\cD(x))  &&\textrm{for all } x\in B_{8R_0}(x_i),\\
  c_1^i(A)& \le C\tha (\cD(x)) &&\textrm{for all } x\in B_{8R_0}(x_i).
\end{aligned}
\ee

\subsection{The hole-filling inequalities}
In this subsection, we derive the hole-filling inequality that follows from the estimate \eqref{hole-start}. However, we split the estimate into two parts. The first one deals with the case when $p(c(x))$ is sufficiently large and the second one for the opposite case. In the following, we keep the notation from the covering introduced in the preceding section. Thus, the first result for large $p$ is the following.
\begin{lem}[Hole-filling inequality I] \label{hole1}
There exists a uniform constant $C$, which is independent of $A$ and $\varepsilon_0$ such that for any $x_i$ being a center of a ball from the covering introduced in Section~\ref{covering}, which fulfills $p(c(x_i))\ge 3$, any $y\in B_{2R_0}(x_i)$ and any $R\in (0, R_0)$  there holds
\begin{equation}\label{casepge2}
\int_{B_R(y)} \tha^{p(c)-2} |\nabla \cD|^2\, dx\le C\left(R^{\nu}+ \int_{A_R(y)}\tha^{p(c)-2} |\nabla \cD|^2\, dx\right),
\end{equation}
where $\nu>0$ comes from \eqref{hole-start}.
\end{lem}
\begin{proof}We use \eqref{hole-start} to get the result. For this purpose, we need to estimate the integral on the right hand side of \eqref{hole-start}.
Assume that the center $x_i$ is fixed such that $p(c(x_i)) \ge 3$. Then from the properties of the covering and from the fact that $\varepsilon_0 <1$, wee see that $p_i^{-}\ge 2$. We also define
\begin{equation}
\eta:= \tha^{\frac{p(c)}{2}} (\cD)\label{def-eta}
\end{equation}
and also to shorten the notation,  we simply write $\tha$ instead of $\tha (\cD)$ in what follows. We start with a simple estimate for arbitrary $x_0 \in B_{R}(x_i)$ and arbitrary $R\in (0, R_0)$. Note that $B_{2R}(x_0)\subset B_{4R_0}(x_i)$.
\begin{equation}\label{MBg1}
\begin{aligned}
&\int_{A_R}\tha^{p(c)-2} \frac{\left|\meanless{\na \vv}{R} \right|^2}{R^2}\,dx \le C\int_{A_R} |\eta - (\eta)_{A_R}|^{\frac{2(p(c)-2)}{p(c)}} \frac{\left|\meanless{\na \vv}{R} \right|^2}{R^2}\,dx \\
&\qquad+\int_{A_R} \left||(\eta)_{A_R}|^{\frac{2(p(c)-2)}{p(c)}}-|(\eta)_{A_R}|^{\frac{2(p_i^{-}(R)-2)}{p_i^{-}(R)}}\right|\frac{\left|\meanless{\na \vv}{R} \right|^2}{R^2}\,dx\\
&\qquad +\int_{A_R}|(\eta)_{A_R}|^{\frac{2(p_i^{-}(R)-2)}{p_i^{-}(R)}}\frac{\left|\meanless{\na \vv}{R} \right|^2}{R^2}\,dx=: I_1 + I_2 + I_3.
\end{aligned}
\end{equation}
Next we take arbitrary $q>1$ and use the H\"{o}lder,  the Poincar\'{e} and the John-Nirenberg inequality to obtain
\begin{equation}\label{MBg2}
\begin{aligned}
I_1&\le R^{-2} \left(\int_{A_R}|\eta - (\eta)_{A_R}|^{\frac{2q'(p(c)-2)}{p(c)}}\,dx\right)^{\frac{1}{q'}} \left(\int_{A_R} \left|\meanless{\na \vv}{R} \right|^{2q}\,dx\right)^{\frac{1}{q}}\\
&\le C(q) R^{-\frac{2}{q}} \left(\int_{A_R}\frac{1+|\eta - (\eta)_{A_R}|^{\frac{2q'(p_i^+ -2)}{p_i^+}}}{R^2}\,dx\right)^{\frac{1}{q'}} \left(\int_{A_R} |\nabla \cD|^{\frac{2q}{q+1}}\,dx\right)^{\frac{q+1}{q}}\\
&\le C(q)  \left(1+\|\eta\|_{BMO}^{\frac{2(p_i^+ -2)}{p_i^+}}\right) \int_{A_R}|\nabla \cD|^{2}\,dx\le C\int_{A_R}\tha^{p(c)-2}|\nabla \cD|^{2}\,dx,
\end{aligned}
\end{equation}
where for that last inequality we used the embedding $W^{1,2}\hookrightarrow BMO$, the uniform bound~\eqref{beta-norm2} and the fact that $p(c)\ge 2$ in $A_R$ and that $\tha\ge 1$.
To estimate $I_2$, we first recall the inequality valid for all $B\ge 1$ and $r\ge s$
\begin{equation}\label{log-est}
B^r - B^s = \int_0^1 \frac{d}{dt} B^{tr + (1-t)s} \, dt = \int_0^1 B^{tr + (1-t)s} \ln B (r-s)   \, dt \le (r-s) B^r \ln B.
\end{equation}
Then using \eqref{log-est}  in $I_2$, together with  the Poincar\'{e} inequality and the fact that $p^+(R)\ge p(c)\ge p^{-}(R)\ge 2$, we obtain 
\begin{equation}\label{MBg4}
\begin{aligned}
I_2&\le  C\left \|(\eta)_{A_R}^{\frac{2(p(c)-2)}{p(c)}}-(\eta)_{A_R}^{\frac{2(p_i^{-}(R)-2)}{p_i^{-}(R)}}\right\|_{L^\infty(A_R)} \int_{A_R} |\nabla \cD|^2\, dx\\
&\le C(\eta)_{A_R}^{2}\left \|\frac{p(c)-2}{p(c)}-\frac{p_i^{-}(R)-2}{p_i^{-}(R)}\right\|_{L^{\infty}(B_{4R}(x_i))} \int_{A_R} \tha^{p(c)-2}|\nabla \cD|^2\, dx\\
&\le C(\eta)_{A_R}^{2}\|p(c)-p_i^-(R)\|_{L^{\infty}(B_{4R}(x_i))} \int_{A_R} \tha^{p(c)-2}|\nabla \cD|^2\, dx
\end{aligned}
\end{equation}
Next, since $\eta \in L^{\beta}$ for any $\beta \in(1,\infty)$ due to \eqref{beta-norm2}, moreover since the function $p$ is Lipschitz and $c\in \mathcal{C}^{0,\frac{\delta}{2+\delta}}$ uniformly with respect to $A$, we have for all $\beta \in [1,\infty)$
$$
(\eta)_{A_R}^{2}\|p(c)-p_i^-(R)\|_{L^{\infty}(B_{4R}(x_i))}\le C(\beta)R^{-\frac{4}{\beta}} R^{\frac{\delta}{2+\delta}}.
$$
Hence, setting $\beta:=4(\delta+2)/\delta$ and substituting the above estimate into \eqref{MBg4}, we get
\begin{equation}\label{MBg5}
\begin{aligned}
I_2&\le  C\int_{A_R} \tha^{p(c)-2}|\nabla \cD|^2\, dx.
\end{aligned}
\end{equation}
Finally, we focus on estimate for $I_3$. First, since $(\eta)_{A_R}$ and $p_i^-(R)$ are constants, we can use the standard Sobolev-Poincar\'{e} inequality and the using the triangle and the H\"{o}lder inequality, we obtain
\begin{equation}\label{MBg3}
\begin{aligned}
I_3&\le \int_{A_R} |(\eta)_{A_R}|^{\frac{2(p_i^{-}(R)-2)}{p_i^{-}(R)}}\frac{\left|\meanless{\na \vv}{R} \right|^2}{R^2}\,dx \le C\left(\int_{A_R} |(\eta)_{A_R}|^{\frac{p_i^{-}(R)-2}{p_i^{-}(R)}}\frac{|\nabla \cD|}{R}\,dx\right)^2\\
&\le C\left(\int_{A_R} |(\eta)_{A_R}-\eta|^{\frac{p_i^{-}(R)-2}{p_i^{-}(R)}}\frac{|\nabla \cD|}{R}\,dx\right)^2+ C\left( \int_{A_R} |\eta|^{\frac{p_i^{-}(R)-2}{p_i^{-}(R)}}\frac{|\nabla \cD|}{R}\,dx\right)^2\\
&\le C\|\eta\|_{BMO}^{\frac{2(p_i^{+}-2)}{p_i^{+}}}\int_{A_R}|\nabla \cD|^2\,dx+ C \int_{A_R} \eta^{\frac{2(p_i^{-}(R)-2)}{p_i^{-}(R)}}|\nabla \cD|^2\,dx\\
&\le C\int_{A_R} \tha^{p(c)-2}|\nabla \cD|^2\,dx,
\end{aligned}
\end{equation}
where for the last inequality we used \eqref{beta-norm2}, the embedding theorem, the John-Nirenberg inequality and the fact that $2\le p_i^- \le p(c)$ in $A_R$. Consequently, substituting estimates \eqref{MBg2}, \eqref{MBg5} and \eqref{MBg3} into \eqref{MBg1} and combining the result with \eqref{hole-start}, we obtain~\eqref{casepge2}.
\end{proof}
The second hole-filling inequality is related to small values of $p(c)$.
\begin{lem}[Hole-filling inequality II] \label{hole1A}
There exists a uniform constant $C$, which is independent of $A$ and $\varepsilon_0$ such that for any $x_i$ being a center of a ball from the covering introduced in Section~\ref{covering}, which fulfills $p(c(x_i))\le 3/2$, any $y\in B_{2R_0}(x_i)$ and any $R\in (0, R_0)$  there holds
\begin{equation}\label{casepge2M}
\int_{B_R(y)} \tha^{p(c)-2} |\nabla \cD|^2\, dx\le C\left(R^{\nu}+ \int_{A_R(y)}\tha^{p(c)-2} |\nabla \cD|^2\, dx\right),
\end{equation}
where $\nu>0$ comes from \eqref{hole-start}.
\end{lem}

\begin{proof}Similarly as before, we just need to estimate the integral on the right hand side of \eqref{hole-start}.
Assume that the center $x_i$ is fixed such that $p(c(x_i)) \le 3/2$. Then from the properties of the covering and from the fact that $\varepsilon_0 <1/2$, we see that $p_i^{+}< 2$. We also recall the definition of $\eta$, see \eqref{def-eta}, i.e., $\eta:= \tha^{\frac{p(c)}{2}} (\cD)$ and since $\eta\ge 1$, it follows from \eqref{beta-norm2} that
\begin{equation}
\|\eta\|_{BMO} + \|\eta^{-1}\|_{BMO}\le C
\label{eta-BMOMB}
\end{equation}
with constant $C$ independent of $A$. Next, we use the H\"{o}lder and the Poincar\'{e} inequality to get (keeping the notation for $A_R$) for arbitrary $q\in (1,\infty)$
\begin{equation}\label{MBg1M}
\begin{aligned}
&\int_{A_R}\tha^{p(c)-2} \frac{\left|\meanless{\na \vv}{R} \right|^2}{R^2}\,dx \le \frac{C}{R^2}\left(\int_{A_R} \tha^{(p(c)-2)q'}\,dx \right)^{\frac{1}{q'}} \left(\int_{A_R} \left|(\nabla \vv)_{0,R}\right|^{2q}\,dx \right)^{\frac{1}{q}}\\
&\le \frac{C}{R^2}\left(\int_{A_R} \tha^{(p(c)-2)q'}\,dx \right)^{\frac{1}{q'}} \left(\int_{A_R} \left|\nabla \cD\right|^{\frac{2q}{1+q}}\,dx \right)^{\frac{q+1}{q}}\\
&\le C\left(\int_{A_R} \frac{\tha^{(p(c)-2)q'}}{R^2}\,dx \right)^{\frac{1}{q'}}\left(\int_{A_R} \frac{\tha^{q(2-p(c))}}{R^2}\,dx \right)^{\frac{1}{q}} \int_{A_R} \tha^{p(c)-2}|\nabla \cD|^2\,dx
\end{aligned}
\end{equation}
Hence, if we show that for some $q$
\begin{equation}\label{allwe}
\left(\int_{A_R} \frac{\tha^{(p(c)-2)q'}}{R^2}\,dx \right)^{\frac{1}{q'}}\left(\int_{A_R} \frac{\tha^{q(2-p(c))}}{R^2}\,dx \right)^{\frac{1}{q}} \le C,
\end{equation}
then following the proof of Lemma~\ref{hole1} we get \eqref{casepge2M}.

To show \eqref{allwe}, we first notice that due to the choice of $\varepsilon_0$ and the fact that $p(c(x_i))\le 3/2$ we have that $p(c)<2$ in $B_{8R}$. Hence, we can estimate the second integral in \eqref{allwe} as follows
\begin{equation*}
\begin{split}
\int_{A_R} \frac{\tha^{q(2-p(c))}}{R^2}\,dx&=\int_{A_R} \frac{\eta^{\frac{2q(2-p(c))}{p(c)}}}{R^2}\, dx \\
&\le C\int_{A_R} \frac{|\eta-(\eta)_{A_R}|^{\frac{2q(2-p(c))}{p(c)}}+|(\eta)_{A_R}|^{\frac{2q(2-p(c))}{p(c)}}}{R^2}\, dx\\
&\le C\left(1+\int_{A_R} \frac{|(\eta)_{A_R}|^{\frac{2q(2-p(c))}{p(c)}}}{R^2}\, dx\right)\le C+C|(\eta)_{A_R}|^{\frac{2q(2-p_i^{-}(R))}{p_i^{-}(R)}},
\end{split}
\end{equation*}
where we used \eqref{eta-BMOMB} and the fact that $\eta\ge 1$. Next, using this estimate in \eqref{allwe} and the facts that $\tha\ge 1$ and $p(c)\le 2$, we have
\begin{equation}\label{allwe2}
\begin{split}
&\left(\int_{A_R} \frac{\tha^{(p(c)-2)q'}}{R^2}\,dx \right)^{\frac{1}{q'}}\left(\int_{A_R} \frac{\tha^{q(2-p(c))}}{R^2}\,dx \right)^{\frac{1}{q}} \\
&\quad \le C\left(\fint_{A_R} \tha^{(p(c)-2)q'}|(\eta)_{A_R}|^{\frac{2q'(2-p_i^{-}(R))}{p_i^{-}(R)}}\,dx \right)^{\frac{1}{q'}}\\
&\quad \le C\left(\fint_{A_R} \tha^{(p(c)-2)q'}\left||(\eta)_{A_R}|^{\frac{2q'(2-p_i^{-}(R))}{p_i^{-}(R)}}-|(\eta)_{A_R}|^{\frac{2q'(2-p(c))}{p(c)}}\right|\,dx \right)^{\frac{1}{q'}}\\
&\qquad +C\left(\fint_{A_R} \tha^{(p(c)-2)q'}|(\eta)_{A_R}|^{\frac{2q'(2-p(c))}{p(c)}}\,dx \right)^{\frac{1}{q'}}\\
&\quad \le C\left(\fint_{A_R} \left||(\eta)_{A_R}|^{\frac{2q'(2-p_i^{-}(R))}{p_i^{-}(R)}}-|(\eta)_{A_R}|^{\frac{2q'(2-p(c))}{p(c)}}\right|\,dx \right)^{\frac{1}{q'}}\\
&\qquad +C\left(\fint_{A_R} \left|\frac{(\eta)_{A_R}}{\eta}\right|^{\frac{2q'(2-p(c))}{p(c)}}\,dx \right)^{\frac{1}{q'}}.
\end{split}
\end{equation}
Then using \eqref{log-est}, \eqref{beta-norm2} and \eqref{cl2d}, we have (compare with the estimates above \eqref{MBg5})
$$
\begin{aligned}
&\fint_{A_R} \left||(\eta)_{A_R}|^{\frac{2q'(2-p_i^{-}(R))}{p_i^{-}(R)}}-|(\eta)_{A_R}|^{\frac{2q'(2-p(c))}{p(c)}}\right|\,dx \le C(p_i^+(R)-p_i^{-}(R))|(\eta)_{A_R}|^{\frac{2q'(2-p^{-})}{p^{-}}+1}\\
&\qquad \le C
\end{aligned}
$$
and using \eqref{eta-BMOMB} and the fact that $\eta \ge 1$ we deduce
$$
\begin{aligned}
\fint_{A_R} \left|\frac{(\eta)_{A_R}}{\eta}\right|^{\frac{2q'(2-p(c))}{p(c)}}\,dx &\le C(q)\fint_{A_R} \left|\frac{(\eta)_{A_R}-\eta}{\eta}\right|^{\frac{2q'(2-p(c))}{p(c)}}\,dx+C(q)\le C(q).
\end{aligned}
$$
Consequently, substituting two above estimates into \eqref{allwe2}, setting for example $q:=2$, we get~\eqref{allwe}. The rest of the proof is the same as the proof of Lemma~\ref{hole1}.
\end{proof}

The last hole-filling inequality, where however the constants will depend on $A$, is related to any  values of $p(c)$ but we need it to deal with the case  $\frac{3}{2}\leq p(c)\leq 3$.
\begin{lem}[Hole-filling inequality III]\label{LemmaEta}Let $y\in B_{2R_0}(x_i)$, where $x_i$ is  a center of a ball from the covering introduced in Section~\ref{covering}. Then for any $R\in (0, R_0)$ there holds
\begin{equation}\label{hole-2}
\int_{B_R(y)} \tha^{p(c)-2} |\nabla \cD|^2\, dx\le C\left(R^{\nu}+ \frac{c_0^i(A)}{c_1^i(A)}\int_{A_R(y)}\tha^{p(c)-2} |\nabla \cD|^2\, dx\right),
\end{equation}
where $\nu>0$ comes from \eqref{hole-start}, $R_0$ from the covering  and $C$ is independent of $A$, $R_0$ and $\varepsilon_0$.
\end{lem}

\begin{proof}
Thus we consider $y\in B_{2R_0}(x_i)$ for some $i$ and consider balls $B_R(y)$ but omit writing $x_0$. We again use \eqref{hole-start} to get the result and all we need is just to  estimate the integral on the right hand side of \eqref{hole-start}. Hence, using \eqref{cA}, \eqref{cAA} and the Poincar\'{e} inequality, we have (we omit writing here the dependence on $A$, which is however hidden in the definition \eqref{cA})
\begin{equation}\label{MBg6}
\begin{aligned}
\int_{A_R}\tha^{p(c)-2} \frac{\left|\meanless{\na \vv}{R} \right|^2}{R^2}\,dx &\le Cc_0^i\int_{A_R}  \frac{\left|\meanless{\na \vv}{R} \right|^2}{R^2}\,dx \le C c_0^i\int_{A_R}|\nabla \cD|^2\, dx \\
&\le \frac{C c_0^i}{c_1^i}\int_{A_R}\tha^{p(c)-2}|\nabla \cD|^2\, dx.
\end{aligned}
\end{equation}
Substituting this inequality into \eqref{hole-start}, we obtain \eqref{hole-2}.
\end{proof}

\subsection{Proof of the main theorem}

We proceed here as follows. In the first step, we show that $\nabla \vv$ is H\"{o}lder continuous but with the modulus of continuity dependent on $A$. In the second step, we however show that we can choose $A$ such that \eqref{dreamA} holds. Consequently, $\vv$ will be a solution to the original problem and therefore belonging to $\mathcal{C}^{1,\mu}$ with some $\mu>0$. Then we can use the standard regularity result for the Stokes system with continuous coefficients to prove the full regularity of solution. Since the last step is quite classical in the theory of PDE's we omit the proof here.
\paragraph{\bf Step 1: Non-uniform $\mathcal{C}^{1,\mu}$ estimates}
We shall start with the following result that will directly imply H\"{o}lder continuity of $\cD$ and consequently also $\nabla \vv$. However, this result will depend on $A$. Nevertheless, this estimate will be used further to obtain the final result. Note that through this section we keep the notation for $c_0^i$ and $c_1^i$ from \eqref{cA} as well as the covering by balls $B_{R_0}(x_i)$.
\begin{lem}[Key estimate]\label{key-est}
There exists uniform constants $C$ and $\nu>0$ independent of $A$ and $\varepsilon_0$ such that for any $x_i$, any $y\in B_{R}(x_i)$ and all $R\in (0,R_0)$, we have
\begin{equation}\label{key1}
\int_{B_R(y)} \tha^{p(c)-2}\left| \na \cD \right|^2\, dx \le CR^{\mu_i}\left(1+R^{-\mu_i}_0\right),
\end{equation}
where $\mu_i$ is given by
\be \label{dfmui}
\mu_i:=\min \left\{\frac{\nu}{2}, \log_2 \left(\frac{c^i_1(A)+Cc^i_0(A)}{Cc^i_0(A)}\right)\right\}
\ee
Furthermore, there exists $\mu_0>0$ independent of $A$ and $\varepsilon_0$ such that if $p(c(x_i))\ge 3$ or $p(c(x_i))\le 3/2$ then \eqref{key1} holds with $\mu_0$ instead of $\mu_i$, where $\mu_0$ is independent of $A$.
\end{lem}

\begin{proof} We omit writing $y$ or $x_i$ in what follows. We define $g:=\tha^{p(c)-2}\left| \na^2 \vv \right|^2$ and we see from  Lemma~\ref{LemmaEta} that it satisfies
\begin{equation}
\int_{B_R}g\, dx \le CR^{\nu} + \frac{Cc_0^i(A)}{c_1(A)} \int_{A_R} g\, dx.
\end{equation}
Thus, we can use Lemma~\ref{HoleFillLem} to obtain (note that the definition \eqref{dfmui} corresponds to \eqref{dfmu})
\be\label{goodestimate}
\int_{B_{R}}\tha^{p(c)-2}\left| \na \cD \right|^2 \,dx \leq R^{\mu_i} \left(1+ \int_{B_{R_0}} \frac{\tha^{p(c)-2}\left| \na \cD \right|^2}{R_0^{\mu_i}}\, dx\right).
\ee
Consequently, using \eqref{est:w12}, we deduce \eqref{key1}. For the second part of the proof, we use the estimates stated in Lemma~\ref{hole1} and Lemma~\ref{hole1A} (here the assumptions $p(c(x_i))\ge 3$ or $p(c(x_i))\le 3/2$ come from) and using again Lemma~\ref{HoleFillLem}, we get \eqref{key1} but with $\mu_0$ independent of $A$.
\end{proof}
From Lemma~\ref{key-est} we can deduce that $\nabla \vv$ is H\"{o}lder continuous. Indeed, defining $\mu:=\min_i \mu_i$ and observing that
$$
\tha^{p(c)-2}\ge (c_1(A))^{p^--2},
$$
we have from \eqref{key1} that for any $y\in \R^2$ and any $R\in (0, R_0)$ (using the point-wise estimate $|\nabla^2 \vv|\le C|\nabla \cD|$) that
$$
\int_{B_R(y)} |\nabla^2 \vv|^2\, dx \le C(A)R^{\mu}.
$$
Thus, using the Morrey embedding, we get that $\vv \in \mathcal{C}^{1,\frac{\mu}{2}}$. However, $\mu$ depends on $A$ and our goal is to show independent estimate.

\paragraph{\bf Step 2: Choice of $A$ such that $\|\cD\|_{\infty}\le A$}
We start with a simple consequence of Lemma~\ref{key-est}.
\begin{lem}[Key estimate for $\tha$]\label{key-esttha}
There exists uniform constants $C$ and $\nu>0$ independent of $A$ and $\varepsilon_0$ such that for any $x_i$, any $y\in B_{R}(x_i)$ and all $R\in (0,R_0)$, we have
\begin{equation}\label{key2}
\int_{B_R(y)} |\nabla \tha^{\frac{p(c)}{2}}|^2\, dx \le CR^{\min\{\mu_i,\frac{\delta}{2(2+\delta)}\}}\left(1+R^{-\mu_i}_0\right),
\end{equation}
where $\mu_i$ is given by \eqref{dfmui} and $\delta$ comes from \eqref{cl2d}.
Furthermore, there exists $\mu_0>$ independent of $A$ and $\varepsilon_0$ such that if $p(c(x_i))\ge 3$ or $p(c(x_i)) \le 3/2$ then \eqref{key2} holds with $\mu_0$ instead of $\mu_i$, where $\mu_0$ is independent of $A$.
\end{lem}

\begin{proof}
First, using the definition of $\tha$ we have
$$
|\nabla \tha^{\frac{p(c)}{2}}|^2 \le C\left(\tha^{p(c-2)}|\nabla \cD|^2 + \tha^{p^++1}|\nabla c|^2\right).
$$
Next, thanks to \eqref{cl2d} and \eqref{beta-norm2}, we can use the H\"{o}lder inequality to get
$$
\int_{B_R}\tha^{p^++1}|\nabla c|^2\, dx \le CR^{\frac{\delta}{2(2+\delta)}}.
$$
Hence, combining these estimates with \eqref{key1} we deduce \eqref{key2}.
\end{proof}

Now, we have everything prepared to  prove the main result of the paper.

\begin{proof}[Proof of the main theorem]
We show, that if $A$ is sufficiently large then \eqref{dreamA} holds true. Hence, assume that $A$ is fixed (but will be chosen later), we  fix $\varepsilon_0:=1/20$ and corresponding covering $B_{R_0/2}(x_i)$. Assume for a contradiction that there is $y\in \Omega$ such that $|\cD (y)|>A$. Due to the properties of covering, we can find $i$ such that $y \in B_{R_0 /2}(x_i)$ and we have that
\begin{equation}\label{FF1}
(1+A)^{\frac{p_i^{-}}{2}} \le \|\tha ^{\frac{p}{2}}\|_{L^{\infty}(B_{R_0/2}(x_i))} \le C \|\tha^{\frac{p}{2}}\|_{\mathcal{C}^{0,\min\{\frac{\mu_i}{2}, \frac{\delta}{4(2+\delta)}\}}(B_{R_0/2}(x_i))}.
\end{equation}
Next, we can use the equivalence of Campanato spaces $\mathcal{L}^{2, 2+\mu}$ with the space of H\"{o}lder continuous functions $\mathcal{C}^{0, \frac{\mu}{2}}$, with the the embedding constant $\dfrac{C(R_0)}{\mu}$, see e.g. \cite{GM},
$$
\begin{aligned}
&\|u\|_{\mathcal{C}^{0, \frac{\mu}{2}}(B_{R_0}(x_i))}\\
&\le \frac{C(R_0)}{\mu}\left(\|u\|^2_{L^2(B_{R_0}(x_i))} + \sup_{x\in B_{R_0}(x_i)} \sup_{R\in (0,R_0)}R^{-2-\mu} \int_{B_R(x)} |u(y)-(u)_{B_R(x)}|^2\, dy \right)^{\frac12}\\
&\le \frac{C(R_0)}{\mu}\left(\|u\|^2_{L^2(B_{R_0}(x_i))} + \sup_{x\in B_{R_0}(x_i)} \sup_{R\in (0,R_0)}R^{-\mu} \int_{B_R(x)} |\nabla u(y)|^2\, dy \right)^{\frac12},
\end{aligned}
$$
where for the second inequality we used the Poincar\'{e} inequality. Thus, applying the above inequality to $\tha^{\frac{p}{2}}$, using the uniform estimate \eqref{beta-norm2} and the estimate in Lemma~\ref{key-esttha}, we find that
$$
\begin{aligned}
&\|\tha^{\frac{p}{2}}\|_{\mathcal{C}^{0, \frac{\min\{\mu_i,\frac{\delta}{2(2+\delta)}\}}{2}}(B_{R_0}(x_i))}\\
&\le \left(\frac{C(R_0,\delta)}{\mu_i}+1\right)\left(\|u\|^2_{1,2} + \sup_{x\in B_{R_0}(x_i)} \sup_{R\in (0,R_0)}R^{-\min\{\mu_i,\frac{\delta}{2(2+\delta)}\}} \int_{B_R(x)} |\nabla \tha^{\frac{p}{2}}(y)|^2\, dy \right)^{\frac12}\\
&\le C(R_0) \left(\frac{1}{\mu_i}+1\right),
\end{aligned}
$$
Thus, inserting this estimate into \eqref{FF1}, we get
\begin{equation}\label{FF2}
(1+A)^{\frac{p_i^{-}}{2}} \le C(R_0) \left(\frac{1}{\mu_i}+1\right).
\end{equation}
Hence, in case we know that $p(c(x_i))\ge 3$ or $p(c(x_i))\le 3/2$ we have that $\mu_i=\mu_0$, where $\mu_0$ is a constant depending only on data. So \eqref{FF2} reduces to
\begin{equation}\label{FF23}
(1+A)^{\frac{p_i^{-}}{2}} \le C.
\end{equation}
In the opposite case, we use the definition of $\mu_i$, see \eqref{dfmui} to get
\begin{equation}\label{FF24}
\begin{split}
(1+A)^{\frac{p_i^{-}}{2}} &\le C(R_0) \left(\frac{1}{\log_2(1+c_i^1(A)/(Cc_0^i))}+1\right)\le C\left(1+ \frac{c_0^i(A)}{c^i_1(A)}\right)\\
&\le C\left(1+A^{\max\{2,p_i^+\}-\min\{2,p_i^{-}\}}\right),
\end{split}
\end{equation}
where we used $\ln(1+x)\ge x/2$ for $x\in (0,1)$ and the definition \eqref{cA}. Since
$$
\max\{2,p_i^+\}-\min\{2,p_i^{-}\}=\left\{
\begin{aligned}
&p_i^+ - p_i^{-} \le \varepsilon_0 &&\textrm{if } p_i^{+}\ge 2,\quad p_i^{-}\le 2,\\
&2 - p_i^{-} \le |p_i^- - 2| &&\textrm{if } p_i^{+}\le 2,\quad p_i^{-}\le 2,\\
&p_i^+ - 2 \le \varepsilon_0 +|p_i^- - 2| &&\textrm{if } p_i^{+}\ge 2,\quad p_i^{-}\ge 2,\\
\end{aligned} \right.
$$
which follows from the fact that $|p_i^+-p_i^-|\le \varepsilon_0$, we have from \eqref{FF24} that
\begin{equation}\label{FF25}
\begin{split}
(1+A)^{\frac{p_i^{-}}{2}}&\le C\left(1+A^{\varepsilon_0 +|p_i^{-}-2|}\right)\le  C\left(1+A\right)^{\frac{p_i^{-}}{2} \frac{2(\varepsilon_0 +|p_i^{-}-2|)}{p_i^{-}}}.
\end{split}
\end{equation}
Finally, since we consider only the case when $p(c(x_i))\in (3/2,3)$ and $\varepsilon_0= 1/20$, we have that $p_i^{-}\in [3/2-\varepsilon_0, 3]$, we deduce
\begin{equation}\label{FF26}
\begin{split}
(1+A)^{\frac{p_i^{-}}{2}}&\le  C_{fin}\left(1+A\right)^{\frac{p_i^{-}}{2} \frac{24}{29}}
\end{split}
\end{equation}
with $C_{fin}$ independent of $A$. Consequently, choosing $A$ so large that
\begin{equation}\label{FF27}
\begin{split}
(1+A)^{\frac{p_i^{-}}{2}\frac{5}{29}}&>  C_{fin},
\end{split}
\end{equation}
we see that \eqref{FF26} cannot be true, and therefore it contradicts the assumption $\|\cD\|_{\infty}\ge A$. Hence, the proof is complete.
\end{proof}

%\bibliography{bibABK}
%\bibliographystyle{amsplain}

\providecommand{\bysame}{\leavevmode\hbox to3em{\hrulefill}\thinspace}
\providecommand{\MR}{\relax\ifhmode\unskip\space\fi MR }
\providecommand{\MRhref}[2]{%
  \href{http://www.ams.org/mathscinet-getitem?mr=#1}{#2}
}
\providecommand{\href}[2]{#2}

\end{document}